\documentclass[10pt]{amsart}
\usepackage[utf8]{inputenc} 
\usepackage[T1]{fontenc} 
\usepackage[english]{babel}

\usepackage{amsfonts,amsmath,amsthm,amssymb,color,mathtools,pdfsync,todonotes,hyperref,mathrsfs,extarrows,bigints,cleveref,enumitem, dirtytalk}
 \usepackage{setspace} 
\usepackage{indentfirst,lettrine}

 \def\N{\mathbb{N}}

\def\elleom#1{L^{#1}(\Omega)}

\def\Sob#1{W^{1,#1}_0}
\def\Sobom#1{W^{1,#1}_0(\Omega)}

\def\Mosco{\overset{M}{\to}}

\numberwithin{equation}{section}
\def\io{\int_\Omega}

\newtheorem{example}{Example}[section]

\newtheorem{dhef}[example]{\sc Definition}
\newtheorem{lemma}[example]{\sc Lemma}
\newtheorem{ohss}[example]{\sc Remark}

\newtheorem{prop}[example]{\sc Proposition}
\newtheorem{theo}[example]{\sc Theorem}

 \newcommand{\abs}[1]{\left|#1\right|}

 \def\car#1{\raise1pt\hbox{$\chi$}_{#1}}

 \def\wli{W^{1,p}_0(\Omega)\cap L^\infty(\Omega)}

 \def\norma#1#2{ \|#1 \|_{\lower 4pt
 \hbox{$\scriptstyle#2$}}}

\title[Mosco-Stability of Natural Growth Obstacle Problems]{Mosco-convergence of convex sets\\ and unilateral problems for differential operators  with lower order terms \\ having natural growth}

\author{    L. Boccardo -    M. A. Palladino -     M. Picerni }

\address{Istituto Lombardo - Sapienza Universit\`a di Roma}
 \email{boccardo@mat.uniroma1.it}
 
\address{Scuola Normale Superiore, Pisa}
\email{mariaantonietta.palladino@sns.it}

\address{SISSA, Trieste}
\email{mpicerni@sissa.it}

\begin{document}

\date{}

\maketitle

\vspace*{-2em}
\begin{flushright}
\small
\textit{Dedicato a Umberto Mosco}
\end{flushright}
\vspace{1em}

\begin{abstract}
We study the stability of solutions to a class of variational inequalities posed on obstacle-type convex sets, under Mosco-convergence. Specifi\\-cally, we consider problems of the form  
\[
\begin{cases}
u \in W_0^{1,p}(\Omega) \cap L^\infty(\Omega), 
& u \geq \psi \quad \text{ in } \Omega, \\[8pt]
\langle A(u), v - u \rangle
+ \displaystyle\int_{\Omega} H(x, u, D u)(v - u) \geq 0,  
& \\[14pt]\forall\; v \in W_0^{1,p}(\Omega) \cap L^\infty(\Omega), \quad &v \geq \psi \quad \text{ in } \Omega.
\end{cases}
\]
\noindent
Here, \( A \) is a Leray–Lions type operator, mapping \( W_0^{1,p}(\Omega) \) into its dual \( W^{-1, p'}(\Omega) \), while \( H(x, u, D u) \) 
grows like \( |D u|^p \). 
The obstacle \( \psi \) is a function in \( W_0^{1,p}(\Omega) \cap L^\infty(\Omega) \).  \noindent
Our main result establishes that the solutions are stable under Mosco-convergence of the constraint sets. This extends classical stability results to natural growth problems.  
\end{abstract}
\tableofcontents

\section{Introduction}

The main focus of the papers \cite{mosco}, \cite{umosco} by Umberto Mosco is the study of the convergence properties of solutions to variational inequalities involving strongly monotone operators. In particular, the author investigates under which conditions on the convex sets the solutions converge, thereby introducing a notion of convergence for sequences of convex sets in reflexive Banach spaces.
Since then, this condition is known as Mosco-convergence and several papers are devoted to this subject.

In a recent paper (see \cite{BoPi1}) the stability result by Mosco has been extended to pseudomonotone operators (which, in general, are not monotone, see \cite{BoDa}) of differential type defined in Sobolev spaces.
	
	On the other hand, monotone operators and convex functionals are strongly related: after the first results by Umberto Mosco on convex functionals, a  paper  dedicated to him (cf. \cite{b-80}) studies the convergence in \( W^{1,p}_0(\Omega) \) of constrained minima on convex sets \( C_n \) for integral functionals of the form  
	$$
	I(v)=\io i(x,Dv)-\io f(x)\,v(x),
	$$
	where \( i(x, \xi) \) is a convex function in \( \xi \), assuming Mosco-convergence 
	(see Definition \eqref{dhef:mosco} below) of the sets \( C_n \). We point out that the proofs do not change if we consider functionals like
	\[
J(v) = \int_{\Omega} j(x, v, Dv) - \int_{\Omega} f(x) v(x),
\]  
where \( j(x, s, \xi) \) is strictly convex in \( \xi \). 

Furthermore, we recall that, in general,  $J$ is convex only if $j(x,s,\xi)$ does not depend on $s$ (see \cite{DacorognaNonConvex}) and that (roughly speaking) its Fr\'echet dfferential $J'$ is of the form
\begin{equation}
\label{j'}
\langle J'(v),w\rangle=\io \partial_\xi j(x,u,Du)\cdot Dw +\io \partial_s j(x,u,Du)w - \io f\,w,
\end{equation}
for all $w\in\wli$.
In particular, the lower order term $\partial_s j(x,u,Du)$ has {\em natural growth}, meaning it has order $p$ with respect to the gradient.

This motivates the study of variational inequalities in $\Sobom{p}$, with\\\noindent $p\in(1,\infty)$, involving a natural growth term. This problem was studied in \cite{BoccardoMurat} and \cite{BoccardoMuratPuel} under the assumption that the convex set was of obstacle type, specifically:
\[
\mathcal{C}(\psi) = \{u \in W_0^{1,p}(\Omega) \mid u \geq \psi \text{ in } \Omega \}.
\]

Building on these results, our aim is to extend the result of \cite{BoPi1}  to differential operators with a Leray–Lions type principal part and lower order terms exhibiting natural growth. More specifically, this paper examines the stability properties of solutions to
\begin{equation}\label{Var_ineq_intro}
	\left\{
	\begin{aligned}
		&u \in W_{0}^{1, p}(\Omega) \cap L^{\infty}(\Omega), \quad u \geq \psi \quad \text{ in } \Omega, \\
		&\langle A(u), v - u \rangle + \int_\Omega a_0(x,u)(v - u) + \int_{\Omega} H(x,u,Du)(v - u) \,   \geq 0, \\
		&\forall\; v \in W_{0}^{1, p}(\Omega) \cap L^{\infty}(\Omega), \quad v \geq \psi \quad \text{ in } \Omega,
	\end{aligned}
	\right.
\end{equation}
where \( A \) is a Leray–Lions type operator and \( H \) has natural growth,
under Mosco-convergence of convex sets.

We point out that problems of the type \eqref{Var_ineq_intro} are loosely related to those studied in \cite{BensoussanFrehseMosco}, where the authors consider a quasilinear quasivariational inequality arising from a stochastic control problem with a quadratic growth Hamiltonian and obstacle.

\section{Main result and outline}
In this section, we present the main result of this work. Before proceeding, let us recall the definition of Mosco-convergence.

\begin{dhef}\label{dhef:mosco}
    We say that a sequence $\{\mathcal C_n\}_n $, of closed convex subsets of a Banach space $X$, Mosco-converges to a closed convex set 
 $\mathcal C_0\subseteq X$ if: 
    \begin{enumerate}[label=\textbf{\arabic*},ref=\arabic*]
        \item\label{Mosco1}  
        For every subsequence $\{v_{n_i}\}_i$, with $ v_{n_i}\in \mathcal C_{n_i} $, weakly convergent to $v_0$, we have $v_0\in \mathcal C_0$.
        \item\label{Mosco2}
        For every $v_0\in \mathcal C_0$, there exist $ v_{n}\in \mathcal C_{n} $ such that 
        $\norma{v_{n}-v_{0}}{}\to0$.
    \end{enumerate}
We will write $\mathcal C_n\Mosco \mathcal C_0$.    
\end{dhef}
In particular, we focus on a specific class of convex sets, known as obstacle-type convex sets.
Given \( \psi \in W_0^{1,p}(\Omega) \cap L^\infty(\Omega) \), we define the associated obstacle-type convex set as
\[
 \mathcal{C}(\psi) = \left\{ v \in W_0^{1,p}(\Omega) : v \geq \psi \text{ a.e. in } \Omega \right\}.
\]
\noindent
We are interested in the stability properties of the solutions to the inequality 
\begin{equation}\label{outline:varineq}
	\left\{
	\begin{aligned}
		& u \in \mathcal{C}(\psi) \cap L^\infty(\Omega), \\
		& \langle A(u), v - u \rangle + \int_\Omega a_0(x,u)(v-u) + \int_\Omega H(x,u,Du)(v - u) \,   \geq 0, \\
		& \forall\; v \in \mathcal{C}(\psi) \cap L^\infty(\Omega),
	\end{aligned}
	\right.
\end{equation}
under Mosco-convergence of the obstacle-type convex sets.
\\
Here, we state the assumptions on $\Omega$, $p$, $A$, $a_0$, $H$ and $\psi$.
\begin{itemize}
    \item \( \Omega \) is an open, bounded subset of \( \mathbb{R}^N \).
    \item \( p \in (1, +\infty) \) and \( p' = \frac{p}{p-1} \).
    \item The principal part \( A \) is a differential operator of second order in divergence form acting from \( W_0^{1,p}(\Omega) \) into \( W^{-1,p'}(\Omega) \):
    \[
    A(v) = -\operatorname{div}(a(x,v,Dv)),
    \]
    where \( a \) is a Carathéodory function with values in \( \mathbb{R}^N \) satisfying the following conditions for all \( s \in \mathbb{R} \), \( \xi, \eta \in \mathbb{R}^N \), and for almost every \( x \in \Omega \):
    \begin{align}
        & a(x, s, \xi) \xi \geq \alpha |\xi|^p, \label{prop:coercivity} \\
        & |a(x, s, \xi)| \leq \beta \left(h(x) + |s|^{p-1} + |\xi|^{p-1}\right), \label{prop:boundedness} \\
        & [a(x, s, \xi) - a(x, s, \eta)] (\xi - \eta) > 0, \label{prop:monotonicity}
    \end{align}
    where \( \alpha, \beta \) are positive constants and \( h \in L^{p'}(\Omega) \).
   
    \item \( a_0 \) is a Carathéodory function with values in \( \mathbb{R} \) such that for all \( s \in \mathbb{R} \), and for almost every \( x \in \Omega \):
    \begin{equation}\label{a_0_boundedness}
        |a_0(x, s)| \leq \beta \left[h(x) + |s|^{p-1} \right],
    \end{equation}
    \begin{equation}\label{prop:a_0}
      a_0(x, s) s \geq \alpha_0 |s|^p, \quad \text{with } \alpha_0 > 0.
    \end{equation}
  
    \item \( H(x, s, \xi) \) is a Carathéodory function such that for all \( s \in \mathbb{R} \), \( \xi \in \mathbb{R}^N \) and for almost every \( x \in \Omega \):
    \begin{equation}\label{H_boundness}
        |H(x, s, \xi)| \leq f(x) + b(|s|) |\xi|^p,
    \end{equation}
    where $f\in\elleom\infty$ is a positive function and \( b(\cdot) \) is a positive, increasing, continuous function defined on \( \mathbb{R}^+ \).
\end{itemize}
\begin{ohss}
    Note that, under assumptions \eqref{prop:coercivity}, \eqref{prop:coercivity}, \eqref{prop:monotonicity}, \( A \) is a bounded, continuous pseudomonotone operator of Leray-Lions type.
\end{ohss}
\noindent
The existence of a solution \( u \in W_0^{1,p}(\Omega) \) to the variational inequalities \eqref{outline:varineq} was proved in \cite{BoccardoMuratPuel}.
 More precisely, the following theorem holds.
\begin{theo}\label{main_BoMuPu}
    There exists a solution \( u \in W_0^{1,p}(\Omega) \) to
    \begin{equation}
    \left\{
    \begin{aligned}
    & u \in \mathcal{C}(\psi) \cap L^\infty(\Omega),  \\
    & \langle A(u), v - u \rangle + \int_\Omega a_0(x,u)(v-u) + \int_\Omega H(x,u,Du)(v - u) \,   \geq 0, \\
    & \forall\; v \in \mathcal{C}(\psi) \cap L^\infty(\Omega).
    \end{aligned}
    \right.
    \end{equation}
\end{theo}
\noindent
We now state the main problem: let \( \psi_n,\, \psi_0 \in W^{1, p}_0(\Omega) \cap L^\infty(\Omega) \). From the argument in the previous section, we know that there exist solutions \( u_n \) and \( u_0 \) to the following variational problems:
\begin{equation}
	\left\{
	\begin{aligned}\label{Mosco:problema:psin}
		&u_n \in \mathcal{C}(\psi_n)\cap L^{\infty}(\Omega),  \\
		&\langle A(u_n), v - u_n \rangle + \int_\Omega a_0(x,u_n)(v - u_n) + \int_{\Omega} H(x, u_n, D u_n)(v - u_n) \,   \geq 0, \\
		&\forall\; v \in \mathcal{C}(\psi_n) \cap L^{\infty}(\Omega)
	\end{aligned}
	\right.
\end{equation}
and
\begin{equation}\label{Mosco:problema:psi0}
	\left\{
	\begin{aligned}
		&u_0 \in\mathcal{C}(\psi_0) \cap L^{\infty}(\Omega), \\
		&\langle A(u_0), v - u_0 \rangle + \int_\Omega a_0(x,u_0)(v - u_0) + \int_{\Omega} H(x, u_0, D u_0)(v - u_0) \,   \geq 0, \\
		&\forall\; v  \in \mathcal{C}(\psi_0)\cap L^{\infty}(\Omega).
	\end{aligned}
	\right.
\end{equation}
\noindent
The main result of this paper states that, if the sequence \( \{\psi_n\}_n  \) is such that convex sets \( \mathcal{C}(\psi_n) \) Mosco-converge to \( \mathcal{C}(\psi_0) \), then the solutions \( u_n \) of \eqref{Mosco:problema:psin} satisfy a stability property. 
That is, the sequence \(\{u_n\}_n\) has a subsequence that converges in \( W^{1,p}_0(\Omega) \) to a solution \( u_0 \) of \eqref{Mosco:problema:psi0}. This is formalized in the following theorem.

\begin{theo}\label{Theoprincipale}
    Let $\psi_n,\, \psi_0\in \Sobom{p}\cap\elleom\infty$. Consider the convex sets
    \[\mathcal{C}(\psi_n)=\{v\in\Sobom{p}\,:\, v \geq \psi_n \,\text{ a.e. in } \Omega\},\]
    \[\mathcal{C}(\psi_0)=\{v\in\Sobom{p}\,:\,v \geq \psi_0\, \text{ a.e. in } \Omega\},\]
    and assume that $\{\psi_n\}_n  $ is bounded in $L^\infty(\Omega)$ and that
    $$
    \mathcal{C}(\psi_n)\;\Mosco\mathcal{C}(\psi_0). 
    $$
    For every $n$, consider the solution $u_n$ to \eqref{Mosco:problema:psin} 
    Then there exists a subsequence $\{u_{n_j}\}_j$ and $u_0\in \mathcal{C}(\psi_0)$ such that $u_{n_j}\to u_0$ in $\Sobom p$. Moreover, $u_0$ is a solution to \eqref{Mosco:problema:psi0}. 
\end{theo}

\smallskip
\noindent\textbf{Outline of the paper:}
in Section \ref{moscoconv}, we recall some results regarding Mosco-convergence of convex sets. In Section \ref{Sect:ripetistime}, we follow the proof from \cite{BoccardoMuratPuel} to establish a quantitative estimate for \( u_n \), the solution to \eqref{Mosco:problema:psin}. In Section \ref{Sect:Mosco}, we prove \Cref{Theoprincipale}. 

For simplicity, in what follows, we will extract subsequences when necessary, without explicitly stating it each time.
\section{Mosco-convergence}\label{moscoconv}
In this section we recall some basic results regarding Mosco-convergence (see \cite{mosco}).
We start with a simple example of the relation between the convergence of convex sets and the convergence of solutions to variational inequalities.

 Let $H$ be a Hilbert space and $\{C_n\}_n$, $C_0$ be closed convex subsets of $H$. For any $f\in H$, let $u_n(f)$ (resp $u_0(f)$) be the projection of $f$ onto the set $C_n$ (resp $C_0$), which are characterized by the inequalities:
 \begin{equation}
 	\label{prn}
 	u_n \in C_n : \, (u_n - f \mid v - u_n) \geq 0, \quad \forall\; v \in C_n;
 \end{equation} 
 \begin{equation}
 	\label{pr0}
 	u_0 \in C_0 : \, (u_0 - f \mid v - u_0) \geq 0, \quad \forall\; v \in C_0.
 \end{equation}
 
\begin{prop}The following are equivalent:
	\begin{itemize}
		\item $C_n\Mosco C_0$,
		\item $\norma{u_n(f) -u_0(f)}{}\to0$, \text{for all }$f \in H$.
	\end{itemize}
\end{prop}
\begin{proof}
We prove that  $\Mosco$ implies $\norma{u_n -u_0}{}\to0$.
\begin{itemize}
	\item Fix $w_0\in C_0$; there exists $w_n\in C_n$ such that $\norma{w_n-w_0}{}\to 0$; the choice $v=w_n$ in \eqref{prn} yields the boundedness (in $H$) of the sequence $\{u_n\}_n$.
	\item
	As a consequence, there exist $u^*\in H$ and a subsequence $\{u_{n_i}\}_i$ weakly convergent to $u^*$; the definition of Mosco-convergence implies that $u^*\in C_0$.
	\item Fix $w_0\in C_0$; there exists $w_n\in C_n$ such that $\norma{w_n-w_0}{}\to 0$; the choice $v=w_n$ in \eqref{prn} yields
	$$
	(u_{n_i} - f \mid w_{n_i} - u_{n_i})\geq0.
	$$
	We pass to the limit (use the weak l.s.c. of the norm) and we deduce that
	$$
	u^*\in C_0:\,(u^* - f \mid w_0 - u^*)\geq0,\;\forall\;w_0\in C_0;
	$$
	that is $u^*$ is a solution on $C_0$: the uniqueness of the solution says that $u^*=u_0$
	and the whole sequence weakly converges to $u_0$.
	\item
	There exists $z_n\in C_n$ such that $\norma{z_n-u_0}{}\to 0$; the choice $v=z_n$ in \eqref{prn} yields
	$$
	0\leq  (u_n -u_0 + u_0 - f \mid z_n-u_0 - u_n+u_0)
	$$
	and
	$$
	(u_n -u_0   \mid     u_n-u_0)\leq  (u_n -u_0   \mid z_n-u_0 ) 
	+( u_0 - f \mid z_n - u_n ),  
	$$
	which implies that $\norma{u_n -u_0}{}\to0$
\end{itemize}
Conversely, we prove that that if $\norma{u_n(f) -u_0(f)}{}\to0$, $\forall\;f\in H$, then  $C_n\Mosco C_0$. 
\begin{itemize}
	\item Let $\{w_n\}_n$ be a sequence weakly convergent to $w_0$, with $w_n\in C_n$.
	Consider the problems \eqref{prn} with $f=w_0$.
	$$
	u_n \in C_n : \, (u_n - w_0 \mid v - u_n) \geq 0, \quad \forall\; v \in C_n.
	$$
	Take $v=w_n$, then
	$$
	u_n \in C_n : \, (u_n - w_0 \mid w_n - u_n) \geq 0.
	$$
	The assumptions imply that it is possible to pass to the limit  and
	$$
	(u_0- w_0 \mid w_0 - u_0) \geq 0,  
	$$
	that is $w_0=u_0$; so that $w_0\in C_0$.
	\item
	Let $w_0\in C_0$. Define $f=w_0$, then $u_0=w_0$.
	Moreover  $\{u_n\}_n$ is the required sequence.
\end{itemize}
\end{proof}

In the specific case of obstacle-type convex sets, we can associate Mosco-convergence with the behavior of the obstacle functions. The following proposition provides sufficient conditions for the Mosco-convergence of these obstacle-type convex sets, which depend on the properties of the associated obstacles (see \cite{b-80}, Proposition 3).

\begin{prop}[Obstacle-Type Convex Sets]\label{ProbObstacleTypeMosco}
Let $\{\psi_n\}_n$ be a sequence of measurable functions such that the obstacle sets
$$ \mathcal{C}(\psi_n) = \{ v \in W_0^{1, p}(\Omega) : v \geq \psi_n  \text{  a.e. in } \Omega\}
$$
are nonempty. Then, $\mathcal{C}(\psi_n)$ Mosco-converges to $\mathcal{C}(\psi_0)$ if any of the following conditions hold:

\begin{enumerate}
\item $\psi_n$ weakly converges in $W_0^{1, p}(\Omega)$ to $\psi_0$, with $\psi_n \leq \psi_0$;
\item $\psi_n$ strongly converges in $W_0^{1, p}(\Omega)$ to $\psi_0$ (see Lemma 1.7 of \cite{mosco});
\item $\psi_n$ weakly converges in $W_0^{1, p}(\Omega)$ to $\psi_0$, with $\psi_n \geq \psi_0$, and $A(\psi_n) \leq T$, where $A(v) = -\operatorname{div}(a(x, D v))$ is a Leray-Lions operator and $T$ belongs to the dual of $W_0^{1, p}(\Omega)$;
\item $\psi_n$ strongly converges in $L^{\infty}(\Omega)$ to $\psi_0$, and there exists $\Psi \in W_0^{1, p}(\Omega)$ such that $\Psi \geq \psi_n$ for all $n$ (see \cite{BoccardoMurat});
\item $\psi_n$ weakly converges in $W_0^{1, q}(\Omega)$ to $\psi_0$, with $q > p$ (see \cite{b-ann}, \cite{BoccardoMurat}).\end{enumerate}
\end{prop}
\noindent
Furthermore, the paper \cite{DalMasoCondMosco} establishes some necessary and sufficient conditions for the Mosco-convergence of sequences of obstacle-type convex sets. 
These conditions are expressed in terms of the properties of the \( p \)-capacities of the level sets of the obstacles.
\noindent 

%We now introduce the main problem: let \( \psi_n,\, \psi_0 \in W^{1, p}(\Omega) \cap L^\infty(\Omega) \). From the argument in the previous section, we know that there exist solutions \( u_n \) and \( u_0 \) to the following variational problems:
%
%\begin{equation}
%\left\{
%\begin{aligned}\label{Mosco:problema:psin}
%    &u_n \in W_{0}^{1, p}(\Omega) \cap L^{\infty}(\Omega), \quad u_n \geq \psi_n \quad \text{in } \Omega, \\
%    &\langle A(u_n), v - u_n \rangle + \int_\Omega a_0(x,u_n)(v - u_n) + \int_{\Omega} H(x, u_n, D u_n)(v - u_n) \,   \geq 0, \\
%    &\forall\; v \in W_{0}^{1, p}(\Omega) \cap L^{\infty}(\Omega), \quad v \geq \psi_n \quad \text{in } \Omega,
%\end{aligned}
%\right.
%\end{equation}
%
%and
%
%\begin{equation}\label{Mosco:problema:psi0}
%\left\{
%\begin{aligned}
%    &u_0 \in W_{0}^{1, p}(\Omega) \cap L^{\infty}(\Omega), \quad u_0 \geq \psi_0 \quad \text{in } \Omega, \\
%    &\langle A(u_0), v - u_0 \rangle + \int_\Omega a_0(x,u_0)(v - u_0) + \int_{\Omega} H(x, u_0, D u_0)(v - u_0) \,   \geq 0, \\
%    &\forall\; v \in W_{0}^{1, p}(\Omega) \cap L^{\infty}(\Omega), \quad v \geq \psi_0 \quad \text{in } \Omega.
%\end{aligned}
%\right.
%\end{equation}

\section{Variational inequalities with a natural growth term}\label{Sect:ripetistime}
In this section we follow the proof of Theorem \ref{main_BoMuPu} presented in \cite{BoccardoMuratPuel} to gain quantitative estimates on the solutions $u$ of
\begin{equation}\label{prob_main}
	\left\{
	\begin{aligned}
		& u \in\mathcal{C}(\psi) \cap L^\infty(\Omega), \\
		& \langle A(u), v - u \rangle + \int_\Omega a_0(x,u)(v-u) + \int_\Omega H(x,u,Du)(v - u) \,   \geq 0, \\
		& \forall\; v \in \mathcal{C}(\psi) \cap L^\infty(\Omega),
	\end{aligned}
	\right.
\end{equation}
where \begin{equation}\label{Convessi_def}
\mathcal{C}= \mathcal{C}(\psi) = \left\{ v \in W_0^{1,p}(\Omega) : v \geq \psi \text{ a.e. in } \Omega \right\}.
\end{equation}
First, note that the solution is not unique in general. So, when we talk about \textit{the} solution, we mean \textit{one} of the possible solutions. 
It is also important to emphasize that the existence of a solution to this problem strongly relies on the obstacle-type structure of these convex sets, and thus this assumption cannot be relaxed.

In the estimates below, for the solution \( u \) of \eqref{prob_main} we explicitly state the dependence on \( \psi \). The goal is to use these estimates to study the stability of the solutions to \eqref{prob_main}, where \( \psi \) is replaced by a sequence \( \{\psi_n\}_n \).

The strategy to obtain estimates for the solution $u$ of \eqref{prob_main} relies on considering a sequence of approximating problems of \eqref{prob_main}, for which the existence of solutions \(u_j\) is known from classical theorems by Brezis (cf. \cite{lions,Bre2}). The next step is to derive an estimate for the sequence \(\{u_j\}_j\) in both \(L^\infty(\Omega)\) and \(W_0^{1,p}(\Omega)\). Thanks to these estimates, we can prove that the sequence \(\{u_j\}_j\) converges to \(u\), the solution to \eqref{prob_main}, which inherits the same bounds.

Consider then the following approximating problems:
 \begin{equation}\label{approximate_system}
    \left\{
    \begin{aligned}
    & u_j \in \mathcal{C}(\psi) \cap L^\infty(\Omega),  \\
    &\begin{aligned}
        \langle A(u_j), v - u_j \rangle + \int_\Omega & a_0(x,u_j)(v-u_j)\\ & + \int_\Omega H_j(x,u_j,Du_j)(v - u_j) \,  \geq 0,
    \end{aligned} \\
        & \forall\; v \in \mathcal{C}(\psi) \cap L^\infty(\Omega),
    \end{aligned}
    \right.
    \end{equation}
where
\begin{equation}\label{dhef:Hj}
H_{j}(x, s, \xi) = \frac{H(x, s, \xi)}{1 + \frac{1}{j} |H(x, s, \xi)|},\quad \forall\;\,j\in \N.
\end{equation}
Recall that, since $H_{j}$ is bounded, for every fixed $j\in \N$ there exists at least one solution $u_{j}$ to the system \eqref{approximate_system} (cf. \cite{Lions}). Furthermore, for every $j\in \N$ there exists a constant $\tilde C_j$ such that 
\[
\norma{u_j}{\elleom\infty}\leq \tilde C_j
\]
(cf. \cite{Bre2}).
%$u_{j}$ belongs to $L^{\infty}(\Omega)$ .
The next Lemma improves this estimate, removing the dependence from $j$ and giving a   bound on the sequence $\{u_j\}_j$ in $\elleom\infty$.

%Following the proof in \cite{BoccardoMuratPuel}, it is possible to derive a uniform bound with respect to \( j \), improving the estimate. In this paper, we provide a more precise version of this estimate, explicitly highlighting the contributions of the various parameters of the problem. Particular attention is given to the role played by the convex set associated with the obstacle \( \psi \), with the specific aim of using these estimates, valid for the solution to the problem \eqref{prob_main}, to study the stability of the solutions to \eqref{prob_main} when the convex set varies under Mosco-convergence.

\begin{lemma}\label{LEMMA:disug:stimainfty}
There exists a positive constant $C_\infty=C_\infty\bigg(\norma{f}{\elleom\infty},\alpha_0,p,\norma{\psi}{\elleom\infty}\bigg)$ such that, for every $j\in\N$, the solutions $u_j$ of \eqref{approximate_system} satisfy the estimate

\begin{equation}\label{stima:infinity_bounded}
\left\|u_{j}\right\|_{L^{\infty}(\Omega)} \leq C_\infty.
\end{equation}
More precisely, $C_\infty$ is defined as
\begin{equation}\label{Def:C_infty}
   C_\infty:= \max\left\{\left(\frac{\norma{f}{\elleom\infty}}{\alpha_{0}}\right)^{1 /(p-1)}, \,\norma{\psi}{\elleom\infty}\right\}.
\end{equation}
\end{lemma}
\noindent
Before proving \Cref{LEMMA:disug:stimainfty}, we need the following Remark.

\begin{ohss}\label{prop:phi_lambda}
	Let $c$, $d$ and $\lambda$ be positive real numbers and \( \varphi_{\lambda} \) be the real function defined by
	\begin{equation}\label{dhef:phi}
		\varphi_{\lambda}(t) = t e^{\lambda t^{2}}.
	\end{equation}
	If $\lambda=c^{2} / 4 d^{2}$ the following inequality holds:
\noindent
\begin{equation}
d \varphi_{\lambda}^{\prime}(t)-c\left|\varphi_{\lambda}(t)\right| \geq d / 2, \quad \forall\; t \in \mathbb{R} .
\end{equation}
\end{ohss}

\begin{proof}[Proof of Lemma \ref{LEMMA:disug:stimainfty}] Consider the function \( v_{j} \) defined by
\begin{equation}
v_{j} = u_{j} + \varphi_{\lambda_j}\left(z_{j}^{-}\right) ,
\end{equation}
where, for any real value $z$, $z^-=\max\{-z\,;\,0\}$, and
$$
\lambda_{j} = \frac{C_{j}^{2}}{4 \alpha^{2}}, \quad C_{j} = b\left(\left\|u_{j}\right\|_{L^{\infty}(\Omega)}\right), \quad z_{j} = u_{j} + \left(\frac{\norma{f}{\elleom\infty}}{\alpha_{0}}\right)^{1 /(p-1)}.
$$
\noindent
    As we already remarked, \( u_{j} \in L^{\infty}(\Omega) \) for each \( j \in \mathbb{N} \), whence \( v_{j} \) belongs to \( \mathcal{C}(\psi) \cap \elleom\infty \) (recall that \( \alpha, \alpha_0 > 0 \)) and can thus be used as a test function in \eqref{approximate_system} to get
    \begin{align*}
\int_{\Omega} a\left(x, u_{j}, D u_{j}\right) D z_{j}^{-} \varphi_{\lambda_{j}}^{\prime}\left({z}_{j}^{-}\right) &+ \int_{\Omega} a_0\left(x, u_{j}\right) \varphi_{\lambda_{j}}\left(z_{j}^{-}\right) \\
&+ \int_{\Omega} H(x,u_j,Du_j) \varphi_{\lambda_{j}}\left(z_{j}^{-}\right) \geq 0.
    \end{align*}
Using \eqref{H_boundness}, we get
\begin{align*}
\int_{\Omega} a\left(x, u_{j}, D u_{j}\right) D z_{j}^{-} \varphi_{\lambda_{j}}^{\prime}\left({z}_{j}^{-}\right) + &\int_{\Omega} a_0\left(x, u_{j}\right) \varphi_{\lambda_{j}}\left(z_{j}^{-}\right) \\
+& \int_{\Omega}\left(\norma{f}{\elleom\infty} + C_{j}\left|D u_{j}\right|^{p}\right) \varphi_{\lambda_{j}}\left(z_{j}^{-}\right) \geq 0,
\end{align*}
which, since $Dz_j=Du_j$ and $\varphi_{\lambda_j}(z_j^-)=0$ a.e. where $z_j\geq0$, \begin{align*}
\int_{\Omega} a\left(x, u_{j}, D z_{j}\right) D z_{j}^{-} \varphi_{\lambda_{j}}^{\prime}\left(z_{j}^{-}\right) +& \int_{\Omega} a_0\left(x, u_{j}\right) \varphi_{\lambda_{j}}\left(z_{j}^{-}\right) \\
+& \int_{\Omega}\left(\norma{f}{\elleom\infty} + C_{j}\left|D z_{j}^{-}\right|^{p}\right) \varphi_{\lambda_{j}}\left(z_{j}^{-}\right) \geq 0.
\end{align*}
This, using \eqref{prop:coercivity} and rearranging terms, leads to
\begin{equation}
\begin{split}
    \int_{\Omega}\left\{\alpha \varphi_{\lambda_{j}}^{\prime}\left(z_{j}^{-}\right) - C_{j} \varphi_{\lambda_{j}}\left(z_{j}^{-}\right)\right\}&\left|D z_{j}^{-}\right|^{p} \\\leq &\int_{\Omega}\left[a_0\left(x, u_{j}\right) + \norma{f}{\elleom\infty}\right] \varphi_{\lambda_{j}}\left(z_{j}^{-}\right).    
\end{split}
\end{equation}
Note that the previous integrals are non-zero only in the set \[ \left\{x \in \Omega:\, u_{j} \leq - \left(\frac{\norma{f}{\elleom\infty}}{\alpha_{0}}\right)^{1 /(p-1)}\right\}, \]
where we have, by \eqref{prop:a_0},
$$
a_{0}\left(x,u_{j}\right) \leq \alpha_{0}\left|u_{j}\right|^{p-2} u_{j} \leq -\norma{f}{\elleom\infty}.
$$
In particular, we have
\begin{equation}
    \int_{\Omega}\left\{\alpha \varphi_{\lambda_{j}}^{\prime}\left(z_{j}^{-}\right) - C_{j} \varphi_{\lambda_{j}}\left(z_{j}^{-}\right)\right\}\left|D z_{j}^{-}\right|^{p} \leq 0.
\end{equation}
It follows (using Remark \ref{prop:phi_lambda} with $d=\alpha$ and $c=C_j$) that we have
$$
\alpha\int_\Omega \left|D z_{j}^{-}\right|^{p} \leq 0,
$$
which implies
$$
z_{j}^{-} = 0.
$$
In conclusion, we have
\begin{equation}\label{stima_dal_basso}
u_{j} \geq -\left(\frac{\norma{f}{\elleom\infty}}{\alpha_{0}}\right)^{1 /(p-1)}.
\end{equation}
\noindent
We now consider the test function \( v_{j} \) defined by
\begin{equation}
v_{j} = u_{j} - \delta_{j} \varphi_{\lambda_{j}}\left(z_{j}^{+}\right),
\end{equation}
where
$$
\begin{aligned}
&
z_{j} = u_{j} - C_{\infty}, \quad C_{\infty} = \max \left\{\left(\frac{\norma{f}{\elleom\infty}}{\alpha_{0}}\right)^{1 /(p-1)} \,;\, \,\norma{\psi}{\elleom\infty}\right\}, \\
& \delta_{j} = e^{ \left\{-\lambda_{j} \|u_{j} \|_{L^{\infty}(\Omega)}^{2}\right\}}\quad \text{and}\quad  z^+=\max\{z\,;\,0\}.
\end{aligned}
$$
Note that \( 0 \leq \delta_{j} \varphi_{\lambda_{j}}\left(z_{j}^{+}\right) \leq z_{j}^{+} \)
and \( v_{j} \in \mathcal{C}(\psi) \).
It follows that, substituting into the inequality \eqref{approximate_system} and using the bound for the function $H$ \eqref{H_boundness}, we have
\begin{equation*}
    \begin{split}
        \delta_{j} \int_{\Omega} a\left(x, u_{j}, D u_{j}\right) D z_{j}^{+} \varphi_{\lambda_{j}}^{\prime}\left(z_{j}^{+}\right) + \delta_{j} \int_{\Omega} a_{0}\left(x, u_{j}\right) \varphi_{\lambda_{j}}\left(z_{j}^{+}\right) \\
\leq \delta_{j} \int_{\Omega}\left(\norma{f}{\elleom\infty} + C_{j}\left|D u_{j}\right|^p\right) \varphi_{\lambda_{j}}\left(z_{j}^{+}\right).
    \end{split}
\end{equation*}
Using \eqref{prop:coercivity} (recall that $\varphi_\lambda'\geq0$), this implies
$$
\int_{\Omega}\left\{\alpha \varphi_{\lambda_{j}}^{\prime}\left(z_{j}^{+}\right) - C_{j} \varphi_{\lambda_{j}}\left(z_{j}^{+}\right)\right\}\left|D z_{j}^{+}\right|^{p} + \int_{\Omega}\left[a_{0}\left(x, u_{j}\right) - \norma{f}{\elleom\infty}\right] \varphi_{\lambda_{j}}\left(z_{j}^{+}\right) \leq 0.
$$
Recall that the previous integrals are non-zero only in \( \left\{x \in \Omega: u_{j} \geq C_{\infty}\right\}.\) In this set, by \eqref{prop:a_0}, we have
$$
a_{0}\left(x, u_{j}\right) \geq \alpha_{0}\left|u_{j}\right|^{p-2} u_{j} \geq \alpha_{0}\left(C_{\infty}\right)^{p-1} \geq \norma{f}{\elleom\infty}.
$$
In particular, we have
\begin{equation}
    \int_{\Omega}\left\{\alpha \varphi_{\lambda_{j}}^{\prime}\left(z_{j}^{+}\right) - C_{j} \varphi_{\lambda_{j}}\left(z_{j}^{+}\right)\right\}\left|D z_{j}^{+}\right|^{p}\leq 0.
\end{equation}
Once again, we apply Remark \ref{prop:phi_lambda} to conclude that \( z_{j}^{+} = 0 \), which implies
\begin{equation}\label{stima_dall_alto}
u_{j} \leq C_{\infty}.  
\end{equation}
\noindent
The inequalities \eqref{stima_dal_basso} and \eqref{stima_dall_alto} give the desired estimate \eqref{stima:infinity_bounded}.\\
\end{proof}
\noindent
The next lemma gives a uniform estimate on the norms of $\{u_j\}_j$ in $\Sobom p$.

\begin{lemma}\label{LEMMA:disug:stimaSob}
    There exists a positive constant $M=M(\alpha,p,\abs{\Omega},\beta)$ (depending only on the data in the assumptions) such that, for every  $\Psi\in\mathcal{C}(\psi)\cap L^\infty(\Omega)$,
   \begin{equation}\label{lem:W_boundness}
        \begin{split}
           \io\left|D u_{j}\right|^{p}
            \leq M(p,\alpha,|\Omega|,\beta)& \bigg[(1+\lambda \tilde C^2)e^{4\lambda \tilde C^2}\bigg]^p\bigg\{\norma{f}{\elleom\infty}+\norma{h}{\elleom{p'}}
            \\
             +\|D\Psi\|_{L^p(\Omega)}^p&+\|D\Psi\|_{L^p(\Omega)} \left(\norma{h}{\elleom{p'}}+C_\infty^{p-1}\right)+C_\infty^{p-1}\bigg\},
        \end{split}
    \end{equation}
\noindent
    with $C_\infty$ given by \eqref{Def:C_infty} and
    \begin{equation} 
    B_\infty=b(C_\infty), \quad \tilde C=C_\infty+\|\Psi\|_{L^\infty(\Omega)},\quad\text{and}\quad\lambda=\frac{B_\infty^{2}}{4 \alpha^{2}}.
    \end{equation}
    % \begin{equation}\label{lem:W_boundness}
    % \left\|u_{j}\right\|_{W_{0}^{1, p}(\Omega)} \leq C_{3} 
    % \end{equation}
\end{lemma}
\begin{proof}
    Let $C_\infty$ be  as in \eqref{Def:C_infty}.
Given $\Psi$ in $\mathcal{C}(\psi)\cap L^\infty(\Omega)$, we define \( v_{j} \) as
\begin{equation}
v_{j}=\left[1-\delta e^{\left\{\lambda\left|u_{j}-\Psi\right|^{2}\right\}}\right] u_{j}+\delta e^{\left\{\lambda\left|u_{j}-\Psi\right|^{2}\right\}} \Psi, 
\end{equation}
where
$$
\lambda=\frac{b(C_\infty)^{2}}{4 \alpha^{2}}, \quad \delta=e^{-\left\{\lambda\left(C_\infty+\|\Psi\|_{L^{\infty}(\Omega)}\right)^{2}\right\}}.
$$
Note that, by \Cref{LEMMA:disug:stimainfty} \( v_{j} \) belongs to \( \mathcal{C}(\psi)\cap L^\infty(\Omega) \). We can thus use it as a test function in \eqref{approximate_system}. Since $v_j-u_j=-\delta\varphi_\lambda(u_j-\Psi)$, we obtain
\begin{equation*}
    \begin{split}
        -\delta \int_{\Omega} a\left(x, u_{j}, D u_{j}\right) &D\left(u_{j} - \Psi\right) \varphi_{\lambda}^{\prime}\left(u_{j} - \Psi\right) \,  \\ 
        &- \delta \int_{\Omega} a_{0}\left(x, u_{j}\right) \varphi_{\lambda}\left(u_{j} - \Psi\right) \,  \\
        &-\delta \int_{\Omega} H_{j}\left(x, u_{j}, D u_{j}\right) \varphi_{\lambda}\left(u_{j} - \Psi\right) \,   \geq 0.
    \end{split}
\end{equation*}
Dividing by $\delta$ both sides and rearranging terms, we get
\begin{equation*}
    \begin{split}
    \int_{\Omega} a\left(x, u_{j}, D u_{j}\right) D u_{j} \varphi_{\lambda}^{\prime}\left(u_{j} - \Psi\right)
    \leq&\int_{\Omega} a\left(x, u_{j}, D u_{j}\right) D \Psi \,\varphi_{\lambda}^{\prime}\left(u_{j} - \Psi\right) \,  \\ 
        &-  \int_{\Omega} a_{0}\left(x, u_{j}\right) \varphi_{\lambda}\left(u_{j} - \Psi\right) \,  \\
        &- \int_{\Omega} H_{j}\left(x, u_{j}, D u_{j}\right) \varphi_{\lambda}\left(u_{j} - \Psi\right).
    \end{split}
\end{equation*}
Using \eqref{H_boundness} and \eqref{prop:coercivity}, we get
\[
\begin{aligned}
\alpha \int_{\Omega} \left|D u_{j}\right|^{p} \varphi_{\lambda}^{\prime}\left(u_{j} - \Psi\right) 
& \leq \int_{\Omega} \left|a_{0}\left(x, u_{j}\right)\right| \left|\varphi_{\lambda}\left(u_{j} - \Psi\right)\right| \\
& \quad + \int_{\Omega} a\left(x, u_{j}, D u_{j}\right) D \Psi \,\varphi_{\lambda}^{\prime}\left(u_{j} - \Psi\right) \\
& \quad + \int_{\Omega} \left(\norma{f}{\elleom\infty} + b(C_\infty) \left|D u_{j}\right|^{p}\right) \left|\varphi_{\lambda}\left(u_{j} - \Psi\right)\right|.
\end{aligned}
\]
That is,
\[
\begin{aligned}
\int_{\Omega} \left|D u_{j}\right|^{p} \bigg[\alpha\varphi_{\lambda}^{\prime}\left(u_{j} - \Psi\right) - & b(C_\infty)|\varphi_{\lambda}\left(u_{j} - \Psi\right)| \bigg]
\\\leq &
\int_{\Omega}\bigg[\left|a_{0}\left(x, u_{j}\right)\right| +\norma{f}{\elleom\infty}\bigg]\left|\varphi_{\lambda}\left(u_{j}-\Psi\right)\right|
\\&
+\int_{\Omega}\left|a\left(x, u_{j}, D u_{j}\right)\right||D \Psi|\left|\varphi_{\lambda}^{\prime}\left(u_{j}-\Psi\right)\right|.
\end{aligned}
\]
By \cref{prop:phi_lambda} (with $d=\alpha$ and $c=b(C_\infty)$) and the choice of \( \lambda \), we have
\begin{equation*}
    \begin{split}
        \frac\alpha2\int_{\Omega}\left|D u_{j}\right|^{p} \leq 
\int_{\Omega}\bigg[\left|a_{0}\left(x, u_{j}\right)\right|& +\norma{f}{\elleom\infty}\bigg]\left|\varphi_{\lambda}\left(u_{j}-\Psi\right)\right|\\
+\int_{\Omega}&\left|a\left(x, u_{j}, D u_{j}\right)\right||D \Psi|\left|\varphi_{\lambda}^{\prime}\left(u_{j}-\Psi\right)\right|,
    \end{split}
\end{equation*}
which implies, by \eqref{prop:boundedness}, \eqref{prop:a_0} and using the fact that $\varphi_\lambda'$ is increasing,
\begin{equation*}
    \begin{split}
        \frac\alpha2\int_{\Omega}\left|D u_{j}\right|^{p}&\leq \left(
 \beta\io\left(h(x)+|u_{j}|^{p-1}\right)
+|\Omega|\norma{f}{\elleom\infty}
\right)
\varphi_\lambda(C_\infty+\|\Psi\|_{L^\infty(\Omega)})
\\&\quad+\varphi_\lambda^\prime(C_\infty+\|\Psi\|_{L^\infty(\Omega)})
\beta\int_\Omega |D\Psi|\left(h(x)+|u_{j}|^{p-1}+|Du_{j}|^{p-1}
\right).
    \end{split}
\end{equation*}

% Thus, by H\"older's and Young's inequalities:\todo[inline]{Guarda il blu sotto}

% $$
% \frac{\alpha}{8} \int_{\Omega}\left|D u_{j}\right|^{p} \leq |\Omega|\left(C_\infty^{p-1}+C_0\right)\varphi_\lambda(2C_\infty)+\varphi_\lambda^\prime(2C_\infty)\|D\Psi\|_p
% \norma{h}{\elleom{p'}}
% +\varphi_\lambda^\prime(2C_\infty)c(p,\alpha)\|D\Psi\|_{L^p(\Omega)}^p.
% $$

% This, choosing \( \Psi \) instead of \( w \), leads to

% $$
% \frac{\alpha}{8} \int_{\Omega}\left|D u_{j}\right|^{p} \leq |\Omega|\left(C_\infty^{p-1}+C_0\right)\varphi_\lambda(2C_\infty)+\varphi_\lambda^\prime(2C_\infty)\|D\Psi\|_p
% \norma{h}{\elleom{p'}}
% \varphi_\lambda^\prime(2C_\infty) c(p,\alpha)\|D\Psi\|_{L^p(\Omega)}^p
% $$
\noindent
Thus, by H\"older's inequality and Lemma \ref{LEMMA:disug:stimainfty}
\[
\begin{split}
    \frac\alpha2\int_{\Omega}\left|D u_{j}\right|^{p}
    \leq 
    &\left[\io \beta \left(h(x)+C_\infty^{p-1}\right)
    +
    |\Omega|\norma{f}{\elleom\infty}
    \right]
    \varphi_\lambda(C_\infty+\|\Psi\|_{L^\infty(\Omega)})\\
    +&\varphi_\lambda^\prime(C_\infty+\|\Psi\|_{L^\infty(\Omega)})
    \beta\norma{D\Psi}{\elleom p}
    \left(\norma{h}{\elleom{p'}}+|\Omega|^\frac{1}{p^\prime}C_\infty^{p-1}\right)\\+
    &
    \varphi_\lambda^\prime(C_\infty+\|\Psi\|_{L^\infty(\Omega)})\beta\int_\Omega |D\Psi||Du_{j}|^{p-1}.
\end{split}
\]
Let $\tilde C=C_\infty+\|\Psi\|_{L^\infty(\Omega)}$. Note that an application of Young's inequality leads to 
\[
\varphi_\lambda^\prime(\tilde C)\beta |D\Psi||Du_{j}|^{p-1}\leq \frac{\alpha}{4}\abs{Du_j}^p+\tilde c(\alpha,p)\bigg[\varphi_\lambda^\prime(\tilde C)\beta |D\Psi|\bigg]^p,
\]
which implies (applying Young's inequality to the other terms) that there exists a positive constant $M=M(\alpha,p,\abs{\Omega},\beta)$ such that
\[\begin{split}
    \frac{\alpha}{4} \int_{\Omega}\left|D u_{j}\right|^{p} \leq & M(p,\alpha,|\Omega|,\beta)\left(C_\infty^{p-1}+\norma{f}{\elleom\infty}+\norma{h}{\elleom{p'}}\right)\varphi_\lambda(\tilde C)\\
    &+\beta\varphi_\lambda^\prime(\tilde C)\|D\Psi\|_p
     \left(\norma{h}{\elleom{p'}}+|\Omega|^\frac{1}{p^\prime}C_\infty^{p-1}\right)
    \\&+\bigg[\beta\varphi_\lambda^\prime(\tilde C)\bigg]^p\tilde c(p,\alpha)\|D\Psi\|_{L^p(\Omega)}^p.
\end{split}\]
\noindent
Now, since 
$$\varphi_\lambda(t)\leq\varphi^\prime_\lambda(t)=(1+2\lambda t^2) e^{\lambda t^2}\quad\text{and}\quad \varphi^\prime_\lambda(t)\geq1,$$
we obtain (possibly changing the constant $M$),
\[
\begin{split}
   \io\left|D u_{j}\right|^{p}
    \leq M(p,\alpha,|\Omega|,\beta)& \bigg[(1+\lambda \tilde C^2)e^{4\lambda \tilde C^2}\bigg]^p\bigg\{\norma{f}{\elleom\infty}+\norma{h}{\elleom{p'}}
    \\
     +\|D\Psi\|_{L^p(\Omega)}^p&+\|D\Psi\|_{L^p(\Omega)} \left(\norma{h}{\elleom{p'}}+C_\infty^{p-1}\right)+C_\infty^{p-1}\bigg\},
\end{split}
\]
which concludes the proof.
\end{proof}

From Lemma \ref{LEMMA:disug:stimainfty} and Lemma \ref{LEMMA:disug:stimaSob}, it follows that there exists \( u \in \mathcal{C}(\psi) \cap L^{\infty}(\Omega) \) and a subsequence (still denoted by \( u_{j} \)) such that:
\begin{equation}
    \begin{array}{ll}
u_{j} \overset{*}{\rightharpoonup} u &  \text{ in } L^{\infty}(\Omega) , \\
u_{j} \rightarrow u & \text{ a.e. in } \Omega, \\
u_{j} \rightharpoonup u & \text{in } W_{0}^{1, p}(\Omega).
\end{array}
\end{equation}
Here, $\rightharpoonup$ denotes the weak convergence and $\overset{*}{\rightharpoonup}$ denotes the weak* convergence in the corresponding spaces.
\noindent
At this point, the authors of \cite{BoccardoMuratPuel} show that it is possible to recover strong convergence in \( W^{1,p}_0(\Omega) \), up to subsequences, of \( \{u_j\}_j \) towards \( u \). By passing to the limit in the approximating problems \eqref{approximate_system}, it then becomes straightforward to conclude that \( u \) is a solution to \eqref{prob_main}. For the details of the proof, we refer to the aforementioned article, specifically Lemma~4, Lemma~5, and the proof of Theorem~1.

For the convenience of the reader, we state here the final result with the precise estimates obtained in  \Cref{LEMMA:disug:stimainfty} and \Cref{LEMMA:disug:stimaSob}.

\begin{lemma}\label{disug_strongcompactness}
The sequence \( \{u_{j}\}_j\) of solutions to \eqref{approximate_system} converges, up to subsequences, in \( W_{0}^{1, p}(\Omega) \). Moreover, if \( u \) is a limit point of \( \{u_{j}\}_j \), then it is a solution to \eqref{prob_main}, and the following estimate holds:
\begin{equation}\label{sol_u:infinity_bounded}
\|u\|_{L^{\infty}(\Omega)} \leq C_\infty, 
\end{equation}
with $C_\infty$ depending continuously (see \eqref{Def:C_infty}) on $\norma{f}{\elleom\infty}$, $\alpha_0$, $p$ and $\norma{\psi}{\elleom\infty}$.
\\
\noindent Moreover, there exists a positive constant \( M = M(\alpha, p, \abs{\Omega}, \beta) \) (depending only on the data in the assumptions) such that, given any \( \Psi \in \mathcal{C}(\psi)\cap L^\infty(\Omega)\),
\begin{equation}\label{sol_u:W_boundness}
    \begin{split}
           \io\left|D u_{j}\right|^{p}
            \leq M(p,\alpha,|\Omega|,\beta)& \bigg[(1+\lambda \tilde C^2)e^{4\lambda \tilde C^2}\bigg]^p\bigg\{\norma{f}{\elleom\infty}+\norma{h}{\elleom{p'}}
            \\
             +\|D\Psi\|_{L^p(\Omega)}^p&+\|D\Psi\|_{L^p(\Omega)} \left(\norma{h}{\elleom{p'}}+C_\infty^{p-1}\right)+C_\infty^{p-1}\bigg\},
        \end{split}
\end{equation}\noindent
 where$$
     \lambda=\frac{b(C_\infty)^2}{4 \alpha^{2}} \quad \text{and}\quad \tilde C=C_\infty+\|\Psi\|_{L^\infty(\Omega)}.
    $$
\end{lemma}

\section{Stability under Mosco-convergence of obstacles}\label{Sect:Mosco}
The aim this section is the proof of the Theorem \ref{Theoprincipale}. From now on, we consider the sequence \( \{\psi_n\}_n \) and the function \( \psi_0 \) satisfying the assumptions of \Cref{Theoprincipale}.

\begin{ohss}
    Observe that in \Cref{Theoprincipale}, we require the sequence \(\{\psi_n\}_n\) to be bounded in \(L^\infty\). In view of point 3 of \Cref{ProbObstacleTypeMosco}, this can be seen as a natural requirement. However, this condition will be needed to apply the techniques from \cite{BoccardoMuratPuel}, since the boundedness results from the previous section, \Cref{LEMMA:disug:stimainfty} and \Cref{LEMMA:disug:stimaSob}, rely on the \(L^\infty\) norm of the obstacle \(\psi\).  
\end{ohss}
From the results obtained in the previous section and summarized in \Cref{disug_strongcompactness}, we derive the following conclusions.
\begin{lemma}\label{mosco:boundLinfty}
    The sequence $\{u_n\}_n$ satisfies the estimate
    \begin{equation}
        \left\|u_{n}\right\|_{L^{\infty}(\Omega)} \leq \max \left\{\left(\frac{\norma{f}{\elleom\infty}}{\alpha_{0}}\right)^{1 /(p-1)}\,; \,\norma{\psi_n}{\elleom\infty}\right\}.
    \end{equation}
    Hence, it is bounded in $\elleom\infty$:
    $$\left\|u_{n}\right\|_{L^{\infty}(\Omega)} \leq C_\infty.$$
    with $$
    C_\infty=C_\infty\left(\norma{f}{\elleom\infty}, \alpha_0, p,\, \sup_n\norma{\psi_n}{\elleom\infty}\right).
    $$
\end{lemma}

\begin{lemma}\label{mosco:boundSobolev}
    The sequence $\{u_n\}_n$ is bounded in $\Sobom p$.
\end{lemma}
The proof of \Cref{mosco:boundSobolev} is an immediate consequence of \Cref{mosco:boundLinfty} and \Cref{disug_strongcompactness}, in view of the following proposition.

\begin{prop}\label{lem:recovery_inf}
    Let $w_0\in \mathcal{C}(\psi_0)\cap\elleom\infty$ and let $w_n\in \mathcal{C}(\psi_n)$ be such that $w_n\to w_0$ in $\Sobom{p}$. Then there exists a sequence
    $\{\tilde w_n\}_n\subset \mathcal{C}(\psi_n)\cap\elleom\infty$ such that $\tilde w_n\to w_0$ in $\Sobom{p}.$ Moreover, if \( \{\psi_n\}_n \) is bounded in \( L^\infty(\Omega) \), so is \( \{\tilde 
w_n\}_n \).

\end{prop}
\begin{proof}
    It suffices to define the truncated functions $\tilde w_n=T_{\lambda_n}(w_n)$ with 
    \[\lambda_n=\max\{\norma{\psi_n}{\elleom\infty}\,;\,\norma{w_0}{\elleom\infty}\}.\]
    With this choice, the function $\tilde w_n$ still belongs to $\mathcal{C}(\psi_n)$. Moreover, we have
    \[
    \norma{ D  w_n- D \tilde w_n}{\elleom p}^p=\int_{\abs{w_n}\geq\lambda_n} \abs{ D  w_n}^p\to0.
    \]
\end{proof}
\begin{proof}(\Cref{mosco:boundSobolev})
    It suffices to apply \Cref{disug_strongcompactness} for each $n$, using the uniform bound of \Cref{mosco:boundLinfty} and choosing an appropriate function $\Psi_n$ in each convex set in the following way.
    By Mosco-convergence and \Cref{lem:recovery_inf}, there exists a sequence $\Psi_n\in\mathcal C(\psi_n)\cap L^\infty(\Omega)$ such that $\Psi_n\to\psi_0$ in $\Sobom p$ and $\{\Psi_n\}_n$ is bounded in $L^\infty(\Omega)$. Hence, the sequence $\{\Psi_n\}_n $ is also bounded in $\Sobom p$. By \Cref{disug_strongcompactness},
    this gives a bound on the $\Sob p$ norm of the sequence $\{u_n\}_n$.
\end{proof}
\noindent
\noindent From Lemma \ref{mosco:boundLinfty} and \ref{mosco:boundSobolev} there exists \( u^* \in W_{0}^{1, \infty}(\Omega) \cap L^{\infty}(\Omega) \) and a subsequence (not relabeled) such that  
\begin{equation} \label{eq:un}
\begin{array}{ll}
u_n \overset{*}{\rightharpoonup} u^* & \text{ in } L^{\infty}(\Omega), \\
u_n \rightarrow u^* & \text{ a.e. in } \Omega,\\
u_n \rightharpoonup u^* & \text{ in } W_{0}^{1, p}(\Omega).
\end{array}
\end{equation}
\noindent
 Moreover, by Mosco-convergence of the convex sets $\mathcal{C}(\psi_n)$, we know that $u^*\in \mathcal{C}(\psi_0)$. 
 
 We now provide the proof of the main result of this work, stated in \Cref{Theoprincipale}. We will show that the sequence \( \{u_n \}_n\), which satisfies \eqref{eq:un}, converges, up to subsequences, in \( W_0^{1, p}(\Omega) \) to a solution \( u^* \) to problem \eqref{Mosco:problema:psi0}.
Before doing this, it is useful to recall the following simple lemma.

\begin{lemma}\label{disug_real_analysis}
Let \( \{y_{j}\}_j  \) a sequence of measurable functions in $\Omega$ such that
\[
y_{j} \geqslant 0, \quad y_{j} \rightarrow y \quad \text{a.e. in } \Omega, \quad y \in L^{1}(\Omega), \quad \int_{\Omega} y_{j} \rightarrow \int_{\Omega} y.
\]
Then
\[
y_{j} \rightarrow y \quad \text{in } L^{1}(\Omega).
\]
\end{lemma}
\begin{proof}
We can write:
\[
\int_{\Omega}\left|y_{j}-y\right| = 2 \int_{0 \leqslant y_{j} \leqslant y}\left(y - y_{j}\right) + \int_{\Omega}\left(y_{j} - y\right).
\]
Observing that 
$$
\left(y - y_{j}\right)\chi_{\{
0 \leqslant y_{j} \leqslant 2y
\}}\leq y\in L^1(\Omega),
$$
 by the Lebesgue Dominated Convergence Theorem, we conclude:
\[
y_{j} \rightarrow y \quad \text{in } L^{1}(\Omega).
\]
\end{proof}
%% serve usare che
\noindent
  We can now prove \Cref{Theoprincipale}.

  \begin{proof}[Proof of \Cref{Theoprincipale}] From \Cref{lem:recovery_inf}, we know that it exists \( {w}_{n} \in \mathcal{C}(\psi_{n}) \cap L^\infty(\Omega) \) such that \( {w}_{n} \to u^* \) in \( W^{1,p}_0(\Omega) \) and $\|w_n\|_{L^\infty(\Omega)}\leq \tilde C.$\\\noindent
  \textbf{Step 1:} We want to prove that 
  \begin{equation}
\int_{\Omega}\left[a\left(x, u_{n}, D u_{n}\right)-a\left(x, u_{n}, D w_{n}\right)\right] D\left(u_{n}-w_{n}\right) \rightarrow 0.
\end{equation}
Let us define a sequence $\{v_{n}\}_n$ as 
\begin{equation}
v_{n}=\left[1-\delta_{n} e^{\left\{\lambda\left|u_{n}-w_{n}\right|^{2}\right\}}\right] u_{n}+\delta_{n} e^{\left\{\lambda\left|u_{n}-w_{n}\right|^{2}\right\}} w_{n} \in \mathcal{C}(\psi_{n}) \cap L^\infty(\Omega),
\end{equation}
where 
$$
B_\infty=b(\sup_n \|u_n\|_{L^{\infty}(\Omega)}),\quad\lambda=\frac{B_\infty^{2}}{4 \alpha^{2}}, \quad 
\delta_{n}=e^ {-\left\{\lambda\left(\|u_{n}\|_{L^{\infty}}+\|w_{n}\|_{L^{\infty}(\Omega)}\right)^{2}\right\}}.
$$
Using \( v_{n} \) as a test function in \eqref{Mosco:problema:psin}, and noting that \( v_{n} - u_{n} = -\delta_{n}\varphi_{\lambda}\left(u_{n} - w_{n}\right) \), we apply \eqref{H_boundness} to obtain
\begin{equation}
    \begin{aligned}
& \int_{\Omega} a\left(x, u_{n}, D u_{n}\right) D\left(u_{n} - w_{n}\right) \varphi_{\lambda}^{\prime}\left(u_{n} - w_{n}\right) 
+ \int_{\Omega} a_0(x,u_{n}) \varphi_{\lambda}\left(u_{n} - w_{n}\right) \\
& \quad \leq \int_{\Omega} \left(\|f\|_{L^\infty(\Omega)} + B_\infty \left|D u_{n}\right|^p\right) \left|\varphi_{\lambda}\left(u_{n} - w_{n}\right)\right|,
\end{aligned}
\end{equation}
which, subtracting 
\[\int_{\Omega} a\left(x, u_{n}, D w_{n}\right) D\left(u_{n} - w_{n}\right) \varphi_{\lambda}^{\prime}\left(u_{n} - w_{n}\right)\]
from both sides and rearranging terms, leads to
\begin{equation}
    \begin{split}
 \int_{\Omega} \bigg[a\left(x, u_{n}, D u_{n}\right)-& a\left(x, u_{n}, D w_{n}\right)\bigg] D\left(u_{n} - w_{n}\right) \varphi_{\lambda}^{\prime}\left(u_{n} - w_{n}\right) \\
\leq&- \int_{\Omega} a_0(x,u_{n}) \varphi_{\lambda}\left(u_{n} - w_{n}\right) 
 \quad \\
 &+ \int_{\Omega} \|f\|_{L^\infty(\Omega)}\left|\varphi_{\lambda}\left(u_{n} - w_{n}\right)\right|
 \\ &+ \int_{\Omega} B_\infty \left|D u_{n}\right|^p\left|\varphi_{\lambda}\left(u_{n} - w_{n}\right)\right|
 \\
 &-\int_{\Omega} a\left(x, u_{n}, D w_{n}\right) D\left(u_{n} - w_{n}\right) \varphi_{\lambda}^{\prime}\left(u_{n} - w_{n}\right).
\end{split}
\end{equation}
Applying property \eqref{prop:coercivity} to the term
\[\int_{\Omega} B_\infty \left|D u_{n}\right|^p\left|\varphi_{\lambda}\left(u_{n} - w_{n}\right)\right|\]
and taking the absolute value (which is then brought inside the integral) in the first term of the right-hand side, we get
$$
\begin{aligned}
 \int_{\Omega} \bigg[a\left(x, u_{n}, D u_{n}\right) &- a\left(x, u_{n}, D w_{n}\right)\bigg] D\left(u_{n} - w_{n}\right) \varphi_{\lambda}^{\prime}\left(u_{n} - w_{n}\right) \\
& \quad \leq -\int_{\Omega} a\left(x, u_{n}, D w_{n}\right) D\left(u_{n} - w_{n}\right) \varphi_{\lambda}^{\prime}\left(u_{n} - w_{n}\right) + \\
& \quad \quad + \int_{\Omega} \bigg[\left|a_0(x,u_{n})\right| + \|f\|_{L^\infty(\Omega)}\bigg] \left|\varphi_{\lambda}\left(u_{n} - w_{n}\right)\right| + \\
& \quad \quad + \frac{B_\infty}\alpha \int_{\Omega} a\left(x, u_{n}, D u_{n}\right) D u_{n} \left|\varphi_{\lambda}\left(u_{n} - w_{n}\right)\right|.
\end{aligned}
$$
Now we subtract
\[\int_{\Omega} 
    \bigg[
    a\left(x, u_{n}, D u_{n}\right) 
    - a\left(x, u_{n}, D w_{n}\right)
    \bigg] 
     D\left(u_{n} - w_{n}\right)\frac{B_\infty}{\alpha} 
    \left| 
    \varphi_{\lambda}\left(u_{n} - w_{n}\right) 
    \right|,\]
from both sides, which leads to
\begin{equation*}
    \begin{split}
            \int_{\Omega} 
    \bigg[
    a\left(x, u_{n}, D u_{n}\right) 
    &- a\left(x, u_{n}, D w_{n}\right)
    \bigg] 
     D\left(u_{n} - w_{n}\right) \\
    &\left[
    \varphi_{\lambda}^{\prime}\left(u_{n} - w_{n}\right) 
    - \frac{B_\infty}{\alpha} 
    \left| 
    \varphi_{\lambda}\left(u_{n} - w_{n}\right) 
    \right| 
    \right]
    \end{split}
\end{equation*}
\begin{equation*}
    \begin{split}
    \quad\leq
-\int_{\Omega} &
a\left(x, u_{n}, D w_{n}\right) 
D\left(u_{n} - w_{n}\right) 
\varphi_{\lambda}^{\prime}\left(u_{n} - w_{n}\right)
\\&+ \frac{B_\infty}{\alpha} 
\int_{\Omega} 
a\left(x, u_{n}, D u_{n}\right) 
D w_{n} 
\left| 
\varphi_{\lambda}\left(u_{n} - w_{n}\right) 
\right|\\
&
+ \frac{B_\infty}{\alpha} 
\int_{\Omega} 
a\left(x, u_{n}, D w_{n}\right) 
D\left(u_{n} - w_{n}\right) 
\left| 
\varphi_{\lambda}\left(u_{n} - w_{n}\right) 
\right| 
\\&+ \int_{\Omega} 
\left[
\left| 
a_{0}\left(x,u_{n}\right)\right|  
+ \|f\|_{L^\infty(\Omega)} 
\right] 
\left| 
\varphi_{\lambda}\left(u_{n} - w_{n}\right) 
\right| .
    \end{split}
\end{equation*}
\noindent
Because of the choice of \( \lambda = B_\infty^{2} / \left(4 \alpha^{2}\right) \), we can apply \Cref{prop:phi_lambda} to infer that the left-hand side is greater than  
$$
\frac{1}{2} \int_{\Omega}\left[a\left(x, u_{n}, D u_{n}\right)-a\left(x, u_{n}, D w_{n}\right)\right] D\left(u_{n}-w_{n}\right).
$$
Moreover, using Lebesgue's dominated convergence Theorem, we see that the right-hand side converges to zero. Hence
\begin{equation}\label{LL:conv_an}
\int_{\Omega}\left[a\left(x, u_{n}, D u_{n}\right)-a\left(x, u_{n}, D w_{n}\right)\right] D\left(u_{n}-w_{n}\right) \rightarrow 0.
\end{equation}
\textbf{Step 2:} We want to show that $\{Du_n\}_n$ converges a.e. to $Du^*$ in $\Omega.$\\\noindent
 Here we follow \cite{Lions}: let \( L_{n} \) be defined by
$$
L_{n} = \left[a\left(x, u_{n}, D u_{n}\right)-a\left(x, u_{n}, D w_n\right)\right] D\left(u_{n}-w_n\right).
$$
Then, by \eqref{prop:monotonicity}, \( L_{n} \) is a positive function, and \eqref{LL:conv_an} implies
$$
L_{n} \rightarrow 0 \quad \text{in } L^{1}(\Omega) \text{ strongly}.
$$
Extracting a subsequence (still denoted by \( \{u_{n}\}_n \) and $\{w_n\}_n$ respectively), we have
$$
\begin{array}{ll}
u_{n} \rightarrow u^* & \text{ a.e. in } \Omega, \\
w_{n} \rightarrow u^* & \text{  a.e. in } \Omega, \\
Dw_{n} \rightarrow Du^* & \text{  a.e. in } \Omega, \\
L_{n} \rightarrow 0 & \text{  a.e. in } \Omega.
\end{array}
$$
Then there exists a subset $Z$ of $\Omega$, of zero measure, such that, for $x \in \Omega\setminus Z$,
$$
|u^*(x)|<\infty, \quad|D u^*(x)|<\infty, \quad|h(x)|<\infty,
$$
and
$$
u_{n}(x) \rightarrow u^*(x), \quad
w_n(x) \rightarrow u^*(x),\quad 
Dw_n(x) \rightarrow Du^*(x),
\quad L_n(x) \rightarrow 0.
$$
Now we fix \( x \in \Omega\setminus Z\) and define
$$
\xi_{n}=Du_{n}(x), \quad \zeta_n=D w_n(x),
$$
so that
$$
L_{n}(x) = \left[a\left(x, u_{n}(x), \xi_{n}\right)-a\left(x, w_{n}(x), \zeta_n\right)\right]\left(\xi_{n}-\zeta_n\right).
$$
By assumption \eqref{prop:coercivity}, we have, after dropping a positive term on the right-hand side
$$
L_{n}(x)\geq \alpha|\xi_n|^p-
a\left(x, u_{n}(x), \xi_{n}\right)\zeta_n-
a\left(x, w_{n}(x), \zeta_n\right)\xi_n.
$$
Since \( \{w_n(x)\}_n \) and \( \{\zeta_n \}_n\) are convergent sequences, there exist two positive constants (depending on \( x \)) \( C_1 \) and \( C_2 \) such that
$$
|w_n(x)|\leq C_1(x),\quad |\zeta_n|\leq C_2(x),
$$
Therefore, using assumption \eqref{prop:boundedness} and the fact that the sequence 
$ \{a\left(x, w_{n}(x), \zeta_n\right)\}$ 
is bounded, we obtain  
$$
L_{n}(x) \geq\alpha\left|\xi_{n}\right|^{p} - c(x)\left[1 + \left|\xi_{n}\right|^{p-1} + \left|\xi_{n}\right|\right],
$$
where \( c(x) \) is a constant which depends on \( x \), but does not depend on \( n \). Since \( L_{n}(x) \rightarrow 0 \), it follows that the sequence \( \{\xi_{n}\}_n \) is bounded. 
Let $\xi^{*}$ be a cluster point of $\xi_{n}$. We have $\left|\xi^{*}\right|<\infty$ and
$$
\left[a\left(x, u^*(x), \xi^{*}\right)-a(x, u^*(x), Du^*(x))\right]\left(\xi^{*}-Du^*(x)\right)=0,
$$
because of the continuity of \( a \) with respect to \( s \) and \( \xi \). By \eqref{prop:monotonicity}, this implies that \( \xi^{*} = Du^*(x) \). It follows that the cluster point is unique (that is, the limit does not depend on the chosen subsequence) and thus
$$
D u_{n}(x) \rightarrow D u^*(x), \quad \forall\; x \in \Omega \setminus Z,
$$  
i.e.  
$$
\;D u_{n} \rightarrow D u^*, \quad \text{ a.e. in } \Omega.
$$
\textbf{Step 3:} We want to prove that $\{Du_n\}_n$ converges to $Du^*$ strongly in $(L^p(\Omega))^N$.\\\noindent
Recall that
\begin{gather}
a\left(x, u_{n}, D u_{n}\right) D u_{n} \geqslant 0, \label{eq:ineq_1} \\
a\left(x, u_{n}, D u_{n}\right) D u_{n} \rightarrow a(x, u^*, D u^*) D u^* \quad \text{a.e. in } \Omega. \label{eq:conv_1}
\end{gather}
Moreover, since the sequence $\{
a(x,u_n,Du_n)\}_n$ is bounded in \((\elleom{p'})^N\) and converges almost everywhere to $a(x,u^*,Du^*)$, we have that 
\[
a\left(x, u_{n}, D u_{n}\right) \to a(x, u^*, D u^*) \quad \text{weakly in } \left(L^{p^{\prime}}(\Omega)\right)^{N} \text{ and a.e. in } \Omega.
\]
This property, together with assumptions \eqref{eq:un} and \eqref{LL:conv_an}, leads us to conclude that
\begin{equation}\label{eq:conv_2}
\int_{\Omega} a\left(x, u_{n}, D u_{n}\right) D u_{n} \rightarrow \int_{\Omega} a(x, u^*, D u^*) D u^*.
\end{equation}
Indeed, by Vitali's Theorem we have
$$
\int_\Omega a(x,u_n,Dw_n)Dw_n\rightarrow  \int_\Omega a(x,u^*,Du^*)Du^*,
$$
$$
\int_\Omega a(x,u_n,Du_n)Dw_n\rightarrow  \int_\Omega a(x,u^*,Du^*)Du^*,
$$
$$
\int_\Omega a(x,u_n,Dw_n)Du\rightarrow  \int_\Omega a(x,u^*,Du^*)Du^*.
$$
Now, by \eqref{eq:ineq_1}, \eqref{eq:conv_1}, \eqref{eq:conv_2} we can take 
$$
y_{n} = a\left(x, u_{n}, D u_{n}\right) D u_{n}, \quad y = a(x, u^*, D u^*) D u^*,
$$
in Lemma \ref{disug_real_analysis} and conclude that 
\begin{equation}
a\left(x, u_{n}, D u_{n}\right) D u_{n} \rightarrow a(x, u, D u^*) D u^* \quad \text { in } L^{1}(\Omega) ,
\end{equation}
Now from assumption \eqref{prop:coercivity} we deduce, using Vitali’s Theorem, that
\begin{equation}
D u_{n} \rightarrow D u^* \quad \text { in }\left(L^{p}(\Omega)\right)^{N} \text { strong. }
\end{equation}
This concludes the proof of the strong convergence of \( u_n \) to \( u^* \) in \( W_0^{1, p}(\Omega) \).\\\noindent
\textbf{Step 4:} 
It remains to show that $u^*$ is a solution to \eqref{Mosco:problema:psi0}.\\\noindent
Let \( v \in \mathcal{C}(\psi_0)\cap L^\infty(\Omega) \) be fixed as a test function in \eqref{Mosco:problema:psi0}. By \Cref{lem:recovery_inf} (recall that the sequence of obstacles $\{\psi_n\}_n$ is bounded in $\elleom\infty$), there exists a sequence \(\{v_n\}_n\), with \( v_n \in \mathcal{C}(\psi_n)\cap\elleom\infty \), such that \( v_n \to v \) in \( W^{1,p}_0(\Omega)\) and weakly* in
\(L^\infty(\Omega) \).
Choosing $v_n$ as a test function in \eqref{Mosco:problema:psin}, we obtain
\begin{equation}\label{Fin_proof}
    \io a\left(x, u_{n}, D u_{n}\right) D(v_n - u_n) + \int_\Omega a_0(x,u_n)(v_n - u_n) + \int_{\Omega} H(x, u_n, D u_n)(v_n - u_n) \,   \geq 0.
\end{equation}
Recall that, by \eqref{eq:un},
\begin{equation}
u_{n}  \to u^* \quad \text{weakly * in } L^{\infty}(\Omega).
\end{equation}
Moreover, using Vitali's Theorem and the inequality \eqref{H_boundness}, we deduce that
\begin{equation}
H\left(x, u_{n}, D u_{n}\right) \to H(x, u^*, D u^*) \quad \text{in } L^{1}(\Omega)
\end{equation}
In addition, by \eqref{prop:boundedness} and \eqref{a_0_boundedness}, we obtain
\begin{equation}
\begin{aligned}
a\left(x, u_{n}, D u_{n}\right) & \to a\left(x, u^*, D u^*\right) \quad \text{in } (L^{p^\prime}(\Omega))^N, \\
a_0\left(x, u_{n}\right) & \to a_0\left(x, u^*\right) \quad \text{in } L^{p^\prime}(\Omega).
\end{aligned}
\end{equation}
We can thus pass to the limit in \eqref{Fin_proof}, obtaining
\begin{equation*}
    \io a\left(x, u^*, D u^*\right) D(v - u^*) + \int_\Omega a_0(x,u^*)(v - u^*) + \int_{\Omega} H(x, u^*, D u^*)(v - u^*) \,   \geq 0.
\end{equation*}
By the arbitrariness of $v$, we conclude that \( u^*\) is a solution to \eqref{Mosco:problema:psi0}. This completes the proof of \Cref{Theoprincipale}.

\end{proof}
\smallskip\centering\textsc{Conflicts of interest}

The authors declare no conflicts of interest.


\begin{thebibliography}{99}
\bibitem{BensoussanFrehseMosco}
A. Bensoussan, J. Frehse, U. Mosco, \emph{A stochastic impulse control problem with quadratic growth Hamiltonian and the corresponding quasi variational inequality}, J. Reine Angew. Math. \textbf{331}, 124–145 (1982).


\bibitem{b-ann}
L. Boccardo, \emph{
An $L^s$-estimate for the gradient of solutions to some nonlinear unilateral problems},
Ann. Mat. Pura Appl. \textbf{141}   277-287 (1985).

\bibitem{b-80}
L. Boccardo, \emph{Some new results about Mosco-convergence}, J. Convex Anal. \textbf{28}, 387-394 (2021).


\bibitem{BoccardoMurat}
L. Boccardo, F. Murat, \emph{Nouveaux résultats de convergence dans des problèmes unilatéraux}, Nonlinear Partial Differential Equations and Their Applications, Collège de France Seminar, Vol. II (Paris, 1979/1980), Res. Notes in Math. 60, Pitman, Boston, Mass.-London, 1982, pp. 64–85, 387–388.

\bibitem{BoccardoMuratPuel}
L. Boccardo, F. Murat, J.P. Puel, \emph{Existence of Bounded Solutions for Non Linear Elliptic Unilateral Problems}, Annali di Matematica Pura ed Applicata, \textbf{152}(1), 183–196 (1988), doi: \href{https://doi.org/10.1007/BF01766148}{10.1007/BF01766148}.

\bibitem{BoDa}
L. Boccardo, B. Dacorogna, \emph{Monotonicity of certain differential operators in divergence form}, Manuscripta Math. \textbf{64}, 253-260 (1989).

\bibitem{BoPi1}
L. Boccardo, M. Picerni \emph{Stability results for variational inequalities in Sobolev spaces under Mosco-convergence of convex sets}, Rend. Mat. Appl. (7), (2025), Online first,  \url{https://www1.mat.uniroma1.it/ricerca/rendiconti/latest/boccardo.pdf}

\bibitem{Bre2}
H. Brézis, \emph{Opérateurs maximaux monotones et semi-groupes de contractions dans les espaces de Hilbert}, North-Holland, Amsterdam, 1973. Available online: \url{https://sites.math.rutgers.edu/~brezis/PUBlications/024.pdf}.

\bibitem{br}
H. Brézis, \emph{Equations et inéquations non linéaires dans les espaces vectoriels en dualité}, Ann. Inst. Fourier (Grenoble) \textbf{18}, 115-175 (1968).

\bibitem{DacorognaNonConvex}
B. Dacorogna, \emph{Convexity of certain integrals of the calculus of variations}, Proceedings of the Royal Society of Edinburgh Section A: Mathematics, \textbf{107}.1, 15-26 (1987).

\bibitem{DalMasoCondMosco}
G. Dal Maso, \emph{Some Necessary and Sufficient Conditions for the Convergence of Sequences of Unilateral Convex Sets}, Journal of Functional Analysis, \textbf{62}(2), 119-159 (1985).

\bibitem{DeGiorgi}
E. De Giorgi, \emph{Semicontinuity theorems in the calculus of variations. With notes by U. Mosco, G. Troianiello, and G. Vergara and a preface by Carlo Sbordone. Dual English-Italian text}, Quaderni dell'Accademia Pontaniana, 56, Accademia Pontaniana, Naples, 2000.

\bibitem{Lions}
J. Leray, J.-L. Lions, \emph{Quelques résultats de Višik sur les problèmes elliptiques non linéaires par les méthodes de Minty-Browder}, Bull. Soc. Math. France, \textbf{93}, 97–107 (1965).

\bibitem{lions}
J.-L. Lions, \emph{Quelques méthodes de résolution des problèmes aux limites non linéaires}, Dunod, Paris; Gauthier-Villars, Paris, 1969.

\bibitem{minty}
G.J. Minty, \emph{Monotone (nonlinear) operators in Hilbert space}, Duke Math. J. \textbf{29}, 341-346 (1962).

\bibitem{mosco}
U. Mosco, \emph{Convergence of convex sets and of solutions to variational inequalities}, Advances in Math. \textbf{3}, 510-585 (1969).

\bibitem{umosco}
U. Mosco, \emph{Approximation of the solutions to some variational inequalities}, Ann. Scuola Norm. Sup. Pisa \textbf{21}, 373-394 (1967).

\bibitem{StampacchiaElliptic}
G. Stampacchia, \emph{Le problème de Dirichlet pour les équations elliptiques du second ordre à coefficients discontinus}, Annales de l'Institut Fourier, \textbf{15}(1), 189–257 (1965).

\end{thebibliography}
\end{document}